\documentclass[12pt]{amsart}
\usepackage{graphicx}
\usepackage[headings]{fullpage}
\usepackage{amssymb,epic,eepic,epsfig,amsbsy,amsmath,amscd}
\numberwithin{equation}{section}
                        \theoremstyle{plain}
\usepackage{mathrsfs}

\newcommand\no[1]{}

\newtheorem{theorem}{Theorem}[section]
\newtheorem{thm}{Theorem}
\newtheorem{lemma}[theorem]{Lemma}

\newtheorem{proposition}[theorem]{Proposition}

\theoremstyle{definition}
\newtheorem{remark}[theorem]{Remark}
\newtheorem{example}[theorem]{Example}
\newtheorem{definition}[theorem]{Definition}

\def\BC{\mathbb C}

\def\BZ{\mathbb Z}

\def\la{\langle}
\def\ra{\rangle}

\DeclareMathOperator{\tr}{\mathrm tr}

\def\be { \begin{equation} }
\def\ee { \end{equation} }

\begin{document}

\title[Reidemeister torsion and Dehn surgery]{Reidemeister torsion and Dehn surgery on twist knots}

\author[Anh T. Tran]{Anh T. Tran}
\address{Department of Mathematical Sciences, University of Texas at Dallas, Richardson, TX 75080, USA}
\email{att140830@utdallas.edu}

\begin{abstract}
We compute the Reidemeister torsion of the complement of a twist knot in $S^3$ and that of the 3-manifold obtained by a Dehn surgery on a twist knot.
\end{abstract}

\thanks{2010 {\em Mathematics Classification:} Primary 57N10. Secondary 57M25.\\
{\em Key words and phrases: Dehn surgery, nonabelian representation, Reidemeister torsion, twist knot.}}

\maketitle

\section{Main results}

In a recent paper Kitano \cite{Ki2015} gives a formula for the Reidemeister torsion of the 3-manifold obtained by a Dehn surgery on the figure eight knot. In this paper we generalize his result to all twist knots. Specifically, we will compute the Reidemeister torsion of the complement of a twist knot in $S^3$ and that of the 3-manifold obtained by a Dehn surgery on a twist knot. 

Let $J(k,l)$ be the link in Figure 1, where $k,l$ denote 
the numbers of half twists in the boxes. Positive (resp. negative) numbers correspond 
to right-handed (resp. left-handed) twists. 
Note that $J(k,l)$ is a knot if and only if $kl$ is even. The knot $J(2,2n)$, where $n \not= 0$, is known as a twist knot. For more information on $J(k,l)$, see \cite{HS}.

\begin{figure}[th]
\centerline{\psfig{file=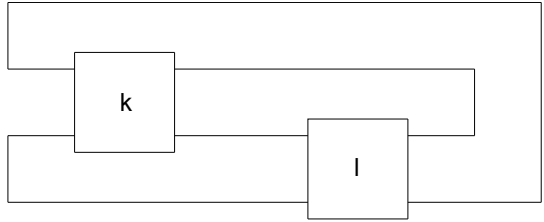,width=3.5in}}
\vspace*{8pt}
\caption{The link $J(k,l)$. }
\end{figure} 

In this paper we fix $K=J(2,2n)$. Let $E_K$ be the complement of $K$ in $S^3$. The fundamental group of $E_K$ has a presentation $\pi_1(E_K)= \la a, b \mid w^na=bw^n \ra$ where $a,b$ are meridians and $w=ba^{-1}b^{-1}a$. A representation $\rho:  \pi_1(E_K) \to SL_2(\BC)$ is called nonabelian if 
the image of $\rho$ is a nonabelian subgroup of $SL_2(\BC)$. Suppose $\rho: \pi_1(E_K) \to SL_2(\BC)$ is a nonabelian representation. Up to conjugation, we may assume that $$\rho(a) = \left[ \begin{array}{cc}
s & 1 \\
0 & s^{-1} \end{array} \right] \quad \text{and} \quad \rho(b) = \left[ \begin{array}{cc}
s & 0 \\
-u & s^{-1} \end{array} \right]$$ where $(s,u) \in (\BC^*)^2$ is a root of the Riley polynomial $\phi_K(s,u)$, see \cite{Ri}.

Let $x :=\tr \rho(a)=s + s^{-1}$ and $z:= \tr \rho(w)=u^2-(x^2-4)u+2$. Let $S_k(z)$ be the Chebychev polynomials of the second kind defined by $S_0(z)=1$, $S_1(z)=z$ and $S_{k}(z) = z S_{k-1}(z) - S_{k-2}(z)$ for all integers $k$. 

\begin{thm}
\label{main1} 
Suppose $\rho: \pi_1(E_K) \to SL_2(\BC)$ is a nonabelian representation. If $x \not= 2$ then the Reidemeister torsion of $E_K$ is given by $$\tau_{\rho}(E_K)=
( 2-x) \frac{S_n(z) - S_{n-2}(z) - 2}{z-2} + x S_{n-1}(z).$$
\end{thm}

Now let $M$ be the 3-manifold obtained by a $\frac{p}{q}$-surgery on the twist knot $K$. The fundamental group $\pi_1(M)$ has a presentation
$$\pi_1(M) = \la a, b \mid w^n a = bw^n, a^p\lambda^q=1 \ra,$$
where $\lambda$ is the canonical longitude corresponding to the meridian $\mu=a$.

\begin{thm} 
\label{main2}
Suppose $\rho: \pi_1(E_K) \to SL_2(\BC)$ is a nonabelian representation which extends to a representation $\rho: \pi_1(M) \to SL_2(\BC)$. If $x \not\in \{0,2\}$ then the Reidemeister torsion of $M$ is given by   
$$\tau_{\rho}(M)=
\left( ( x-2) \frac{S_n(z) - S_{n-2}(z) - 2}{z-2} - x S_{n-1}(z) \right) \left( u^{-2}(u+1)(x^2-4) - 1 \right)  x^{-2}.$$
\end{thm}

\begin{remark} (1) One can see that the expression $(S_n(z) - S_{n-2}(z) - 2)/(z-2)$ is actually a polynomial in $z$.

(2) Theorem \ref{main2} generalizes the formula for the Reidemeister torsion of the 3-manifold obtained by a $\frac{p}{q}$-surgery on the figure eight knot by Kitano \cite{Ki2015}.
\end{remark}

\begin{example} 
(1) If $n=1$, then $K=J(2,2)$ is the trefoil knot. In this case the Riley polynomial is $\phi_K(s,u) = u-(x^2-3)$, and hence $$\tau_{\rho}(M)=-2 \left( u^{-2}(u+1) (x^2-4) - 1 \right)  x^{-2}=\frac{2}{x^2(x^2-3)^2}.$$

(2) If $n=-1$, then $K=J(2,-2)$ is the figure eight knot. In this case the Riley polynomial is $\phi_K(s,u) =u^2 -(u+1)(x^2-5)$, and hence $$\tau_{\rho}(M)=(2x-2)\left( u^{-2}(u+1)(x^2-4) - 1 \right)  x^{-2}=\frac{2x-2}{x^2(x^2-5)}.$$
\end{example}

The paper is organized as follows. In Section \ref{section-chev} we review the Chebyshev polynomials of the second kind and their properties. In Section \ref{nab} we give a formula for the Riley polynomial of a twist knot, and compute the trace of a canonical longitude. In Section \ref{section-R} we review the Reidemeister torsion of a knot complement and its computation using Fox's free calculus. We prove Theorems \ref{main1} and \ref{main2} in Section \ref{section-proof}.

\section{Chebyshev polynomials}

\label{section-chev}

Recall that $S_k(z)$ are the Chebychev polynomials defined by $S_0(z)=1$, $S_1(z)=z$ and $S_{k}(z) = z S_{k-1}(z) - S_{k-2}(z)$ for all integers $k$. The following lemma is elementary.

\begin{lemma} \label{chev} One has $S^2_k(z) - z S_k(z) S_{k-1}(z) + S^2_{k-1}(z)=1.$
\end{lemma}

Let $P_k(z) := \sum_{i=0}^k S_i(z)$. 

\begin{lemma}
\label{P_k}
One has $P_k(z) = \frac{S_{k+1}(z)-S_{k}(z)-1}{z-2}.$
\end{lemma}

\begin{proof}
We have 
\begin{eqnarray*}
z P_k(z) &=& z \sum_{i=0}^k S_i(z) = \sum_{i=0}^k \big( S_{i+1}(z) + S_{i-1}(z) \big) \\
&=& \big( P_k(z) + S_{k+1}(z)- S_0(z) \big) + \big( P_k(z) - S_{k}(z) + S_{-1}(z) \big)\\
&=& 2 P_k(z) + S_{k+1}(z)-S_{k}(z)-1.
\end{eqnarray*}
The lemma follows.
\end{proof}

\begin{lemma}
\label{P^2_k}
One has
$P^2_k(z)+P^2_{k-1}(z) - z P_k(z) P_{k-1}(z) = P_k(z) + P_{k-1}(z).$
\end{lemma}

\begin{proof}
Let $Q_k(z) = \big( P^2_k(z)+P^2_{k-1}(z) - z P_k(z) P_{k-1}(z) \big) - \big( P_k(z) + P_{k-1}(z) \big).$
We have $$Q_{k+1}(z) - Q_k(z) = \big( P_{k+1}(z) - P_{k-1}(z) \big) \big( P_{k+1}(z) + P_{k-1}(z) - z P_k(z) -1 \big).$$
Since $z P_k(z) = \sum_{i=0}^k \big( S_{i+1}(z) + S_{i-1}(z) \big) = P_{k+1}(z)-1 + P_{k-1}(z)$, we obtain $Q_{k+1}(z) = Q_k(z)$ for all integers $k$. Hence $Q_k(z)=Q_1(z)=0$. 
\end{proof}

\begin{proposition}
\label{formulas}
Suppose $V = \left[ \begin{array}{cc}
a & b \\
c & d \end{array} \right] \in SL_2(\BC)$. Then 
\begin{eqnarray}
V^k &=& \left[ \begin{array}{cc}
S_{k}(t) - d S_{k-1}(t) & b S_{k-1}(t) \\
c S_{k-1}(t) & S_{k}(t) - a S_{k-1}(t) \end{array} \right], \label{power}\\
\sum_{i=0}^k V^i &=& \left[ \begin{array}{cc}
P_{k}(t) - d P_{k-1}(t) & b P_{k-1}(t)\\
c P_{k-1}(t) & P_{k}(t) - a P_{k-1}(t) \end{array} \right], \label{sum-power}
\end{eqnarray}
where $t:= \tr V = a+d$. Moreover, one has
\begin{equation} 
\label{det-sum}
\det \left( \sum_{i=0}^k V^i \right) = \frac{S_{k+1}(z) - S_{k-1}(z)-2}{z-2}.
\end{equation}
\end{proposition}

\begin{proof}
Since $\det V=1$, by the Cayley-Hamilton theorem we have $V^2-t V+I=0$. This implies that $V^k - t V^{k-1} + V^{k-2}=0$ for all integers $k$. Hence, by induction on $k$, one can show that $V^k=S_{k}(t) I - S_{k-1}(t) V^{-1}$. Since $V^{-1} = \left[ \begin{array}{cc}
d & -b \\
-c & a \end{array} \right]$, \eqref{power} follows.

Since $P_k(t) = \sum_{i=0}^k S_i(t)$, \eqref{sum-power} follows directly from \eqref{power}. By Lemma \ref{P^2_k} we have
\begin{eqnarray*}
\det \left( \sum_{i=0}^k V^i \right) &=& P^2_k(t)+(ad-bc)P^2_{k-1}(t)-(a+d)P_k(t)P_{k-1}(t)\\
&=& P^2_k(t)+P^2_{k-1}(t)-tP_k(t)P_{k-1}(t)\\
&=& P_k(t) + P_{k-1}(t).
\end{eqnarray*}
Then \eqref{det-sum} follows from Lemma \ref{P_k}.
\end{proof}

\section{Nonabelian representations} 

\label{nab}

In this section we give a formula for the Riley polynomial of a twist knot. This formula was already obtained in \cite{DHY, Mo}. We also compute the trace of a canonical longitude.

\subsection{Riley polynomial} Recall that $K = J(2,2n)$ and $E_K  = S^3 \setminus K$.  The fundamental group of $E_K$ has a presentation $\pi_1(E_K) = \la a, b \mid w^n a = b w^n\ra$ where $a,b$ are meridians and $w=ba^{-1}b^{-1}a$.  Suppose $\rho: \pi_1(E_K) \to SL_2(\BC)$ is a nonabelian representation. Up to conjugation, we may assume that $$\rho(a) = \left[ \begin{array}{cc}
s & 1 \\
0 & s^{-1} \end{array} \right] \quad \text{and} \quad \rho(b) = \left[ \begin{array}{cc}
s & 0 \\
-u & s^{-1} \end{array} \right]$$ where $(s,u) \in (\BC^*)^2$ is a root of the Riley polynomial $\phi_K(s,u)$.

We now compute $\phi_K(s,u)$. Since $$\rho(w) = \left[ \begin{array}{cc}
1-s^2u & s^{-1}-s-su \\
(s-s^{-1})u+su^2 & 1+(2-s^{-2})u+u^2 \end{array} \right],$$ by Lemma \ref{formulas} we have $$\rho(w^n) = \left[ \begin{array}{cc}
S_{n}(z) - \big( 1+(2-s^{-2})u+u^2 \big) S_{n-1}(z) & (s^{-1}-s-su) S_{n-1}(z) \\
\big( (s-s^{-1})u+su^2 \big) S_{n-1}(z) & S_{n}(z) - (1-s^2u) S_{n-1}(z) \end{array} \right],$$
where $z = \tr \rho(w)=2+(2-s^2-s^{-2})u+u^2$. Hence, by a direct computation we have
$$\rho(w^n a -  b w^n) = \left[ \begin{array}{cc}
0 & \phi_K(s,u) \\
u \phi_K(s,u) & 0\end{array} \right]$$ where
$$\phi_K(s,u) = S_n(z) - \left( u^2 -(u+1)(s^2+s^{-2}-3) \right)S_{n-1}(z).$$

\subsection{Trace of the longitude} It is known that the canonical longitude corresponding to the meridian $\mu=a$ is $\lambda=\overleftarrow{w}^n w^n$, where $\overleftarrow{w}$ is the word in the letters $a,b$ obtained by writing $w$ in the reversed order. We now compute its trace. This computation will be used in the proof of Theorem \ref{main2}.

\begin{lemma} 
\label{S^2}
One has
$S^2_{n-1}(z)= \frac{1}{( u + 2 - s^2 - s^{-2}) \left( u^2-  (s^2+s^{-2}-2)(u+1) \right)}.$
\end{lemma}

\begin{proof}
Since $(s,u) \in (\BC^*)^2$ is a root of the Riley polynomial $\phi_K(s,u)$, we have $S_n(z) = \left( u^2 -(u+1)(s^2+s^{-2}-3) \right)S_{n-1}(z)$. Lemma \ref{chev} then implies that
\begin{eqnarray*}
1 &=& S^2_n(z) - z S_n(z) S_{n-1}(z) + S^2_{n-1}(z) \\
&=& \Big( \left( u^2 -(u+1)(s^2+s^{-2}-3) \right)^2 - z \left( u^2 -(u+1)(s^2+s^{-2}-3) \right) + 1 \Big) S^2_{n-1}(z).
\end{eqnarray*}
By replacing $z=2+(2-s^2-s^{-2})u+u^2$ into the first factor of the above expression, we obtain the desired equality.
\end{proof}

\begin{proposition}
\label{longitude}
One has $\tr\rho(\lambda) -2 = \frac{u^2(s^2+s^{-2}+2)}{(u+1)(s^2+s^{-2}-2)-u^2}.$
\end{proposition}

\begin{proof}
By Lemma \ref{formulas} we have $$\rho(w^n) = \left[ \begin{array}{cc}
S_{n}(z) - \big( 1+(2-s^{-2})u+u^2 \big) S_{n-1}(z) & (s^{-1}-s-su) S_{n-1}(z)\\
\big( (s-s^{-1})u+su^2 \big) S_{n-1}(z) & S_{n}(z) - (1-s^2u) S_{n-1}(z) \end{array} \right].$$ 
Similarly,
$$\rho(\overleftarrow{w}^n) = \left[ \begin{array}{cc}
S_{n}(z) - (1-s^{-2}u) S_{n-1}(z) & (s-s^{-1}-s^{-1}u) S_{n-1}(z)\\
\big( (s^{-1}-s)u+s^{-1}u^2 \big) S_{n-1}(z) & S_{n}(z) - \big( 1+(2-s^2)u+u^2 \big) S_{n-1}(z) \end{array} \right].$$
Hence, by a direct calculation we have
\begin{eqnarray*}
\tr \rho(\lambda) &=& \tr (\rho(\overleftarrow{w}^n)\rho(w)) \\
&=& 2 S^2_{n}(z) - 2 z S_{n}(z) S_{n-1}(z) + \big( 2 + (s^4+s^{-4}-2) u^2 - (s^2+s^{-2}+2)u^3\big) S^2_{n-1}(z)\\
&=& 2 + u^2 (s^2+s^{-2}+2) ( s^2+s^{-2}-2 - u ) S^2_{n-1}(z).
\end{eqnarray*}
The lemma then follows from Lemma \ref{S^2}. 
\end{proof}

\section{Reidemeister torsion} 

\label{section-R}

In this section we briefly review the Reidemeister torsion of a knot complement and its computation using Fox's free calculus. For more details on the Reidemeister torsion, see \cite{Jo, Mi61, Mi62, Mi66, Tu}.

\subsection{Torsion of a chain complex}

Let $C$ be a chain complex of finite dimensional vector spaces over $\BC$:
$$
C = \left( 0 \to C_m \stackrel{\partial_m}{\longrightarrow} C_{m-1} 
\stackrel{\partial_{m-1}}{\longrightarrow} \cdots \stackrel{\partial_{2}}{\longrightarrow} 
C_1 \stackrel{\partial_{1}}{\longrightarrow} C_0 \to 0\right)
$$
such that for each $i=0,1, \cdots, m$ the followings hold
\begin{itemize}
\item the homology group $H_i(C)$ is trivial, and
\item a preferred basis $c_i$ of $C_i$ is given.
\end{itemize}

Let $B_i\subset C_i$ be the image of $\partial_{i+1}$. For each $i$ 
choose a basis $b_i$ of $B_i$. The short exact sequence of $\BC$-vector spaces
$$
0 \to B_{i} \longrightarrow C_i \stackrel{\partial_i}{\longrightarrow} B_{i-1} \to 0
$$
implies that  a new basis of $C_i$ can be obtained by taking the union of the vectors of $b_i$
and some lifts $\tilde{b}_{i-1}$ of the vectors $b_{i-1}$. Define $[(b_i \cup \tilde{b}_{i-1})/c_i]$ to be the determinant of the matrix expressing $(b_i \cup \tilde{b}_{i-1})$ in the basis $c_i$. Note that this scalar does not depend on the choice of the lift $\tilde{b}_{i-1}$ of $b_{i-1}$.

\begin{definition} The \emph{torsion} of $C$ is defined to be
$$
\tau(C) := \prod_{i=0}^m \ [(b_i \cup \tilde{b}_{i-1})/c_i]^{(-1)^{i+1}} \ \in \BC\setminus\{0\}.
$$
\end{definition}

\begin{remark} Once a preferred basis of $C$ is given, $\tau(C)$ is independent of the choice of $b_0,\dots,b_m$.
\end{remark}

\subsection{Reidemeister torsion of a CW-complex} Let $M$ be a finite CW-complex and $\rho: \pi_1(M) \to SL_2(\BC)$ a representation. Denote by $\tilde{M}$ the universal covering of $M$. The fundamental group $\pi_1(M)$ acts on $\tilde{M}$ as deck transformations. Then the chain complex $C(\tilde{M}; \BZ)$ has the structure of a chain complex of left $\BZ[\pi_1(M)]$-modules. 

Let $V$ be the 2-dimensional vector space $\BC^2$ with the canonical basis $\{e_1, e_2\}$. Using the representation $\rho$, $V$ has the structure of a right $\BZ[\pi_1(M)]$-module which we denote by $V_{\rho}$. Define the chain complex $C(M; V_{\rho})$ to be $C(\tilde{M}; \BZ) \otimes_{\BZ[\pi_1(M)]} V_{\rho}$, and choose a preferred basis of $C(M; V_{\rho})$ as follows. Let $\{u^i_1, \cdots, u^i_{m_i}\}$ be the set of $i$-cells of $M$, and choose a lift $\tilde{u}^i_j$ of each cell. Then 
$\{ \tilde{u}^i_1 \otimes e_1, \tilde{u}^i_1 \otimes e_2, \cdots, \tilde{u}^i_{m_i} \otimes e_1, \tilde{u}^i_{m_i} \otimes e_2\}$ is chosen to be the preferred basis of $C_i(M; V_{\rho})$.

A representation $\rho$ is called \textit{acyclic} if all the homology groups $H_i(M; V_{\rho})$ are trivial. 

\begin{definition}
The Reidemeister torsion $\tau_{\rho}(M)$ is defined as follows:
$$\tau_{\rho}(M)=\begin{cases} \tau(C(M; V_{\rho})) &\mbox{if } \rho \mbox{ is acyclic}, \\ 
0 & \mbox{otherwise}. \end{cases} $$
\end{definition}

\subsection{Reidemeister torsion of a knot complement and Fox's free calculus}

Let $L$ be a knot in $S^3$ and $E_L$ its complement. We choose a Wirtinger presentation for the fundamental group of $E_L$:
$$
\pi_1(E_L)=
\langle a_1,\ldots,a_l~|~r_1,\ldots,r_{l-1}\rangle.
$$
Let 
$\rho: \pi_1(E_L) \to SL_2(\BC)$ be a representation. This map induces a ring homomorphism $\rho: {\BZ}[\pi_1(E_L)]\to M_2(\BC)$, 
where ${\BZ}[\pi_1(E_L)]$ is the group ring of $\pi_1(E_L)$ 
and 
$M_2(\BC)$ is the matrix algebra of degree $2$ over ${\BC}$. 
Consider the $(l-1)\times l$ matrix $A$ 
whose $(i,j)$-component is the $2\times 2$ matrix 
$$
\rho\left(\frac{\partial r_i}{\partial a_j}\right)
\in M_2(\BC),
$$
where 
${\partial}/{\partial a}$ 
denotes the Fox calculus. 
For 
$1\leq j\leq l$, 
denote by $A_j$ 
the $(l-1)\times(l-1)$ matrix obtained from $A$ 
by removing the $j$th column. 
We regard $A_j$ as 
a $2(l-1)\times 2(l-1)$ matrix with coefficients in 
$\BC$. Then Johnson showed the following.

\begin{theorem} \cite{Jo} 
\label{Johnson}
Let $\rho: \pi_1(E_L) \to SL_2(\BC)$ be a representation such that $\det(\rho(a_1)-I) \not= 0$. Then the Reidemeister torsion of $E_L$ is given by
$$
\tau_\rho(E_L)
=\frac{\det A_1}{\det(\rho(a_1)-I)}. 
$$
\end{theorem}

\section{Proof of main results}

\label{section-proof}

\subsection{Proof of Theorem \ref{main1}} We will apply Theorem \ref{Johnson} to calculate the Reidemeister torsion of the complement $E_K$ of the twist knot $K=J(2,2n)$. 

Recall that $\pi_1(E_K)=\la a,b \mid w^n a = b w^n\ra$. We have $\det (\rho(b)-I) = 2- (s+s^{-1})=2-x$. Let $r=w^naw^{-n}b^{-1}$. By a direct computation we have 
\begin{eqnarray*}
\frac{\partial r}{\partial a} &=& w^n \left( 1+(1-a)(w^{-1}+ \cdots + w^{-n})\frac{\partial w}{\partial a} \right)\\
&=& w^n \left( 1+(1-a)(1+ w^{-1} + \cdots + w^{-(n-1)})a^{-1}(1-b) \right) .
\end{eqnarray*}
Suppose $x \not=2$. Then $\det (\rho(b)-I) \not= 0$ and hence $$\tau_{\rho}(E_K) = \det \rho \left( \frac{\partial r}{\partial a} \right) \big/ \det (\rho(b)-I) = \det \rho \left( \frac{\partial r}{\partial a} \right) \big/ (2-x).$$.

Let $\Delta = \rho(1+ w^{-1} + \cdots + w^{-(n-1)})$ and $\Omega = \rho \big( a^{-1}(1-b)(1-a)\big) \Delta$. Then $$\det \rho \left( \frac{\partial r}{\partial a} \right) = \det (I + \Omega) = 1 + \det \Omega + \tr \Omega.$$

\begin{lemma}
\label{det-O}
One has $\det \Omega = (2-x)^2 \left( \frac{S_{n}(z) - S_{n-2}(z)-2}{z-2} \right)$.
\end{lemma}

\begin{proof}
Since $\tr\rho(w^{-1}) = \tr \rho(w)=z$, by Lemma \ref{formulas} we have $\det \Delta = \frac{S_{n}(z) - S_{n-2}(z)-2}{z-2}$. The lemma follows, since $\det \Omega = \det \rho(a^{-1}(1-a)(1-b)) \det \Delta = (2-x)^2 \det \Delta$.
\end{proof} 

\begin{lemma}
\label{trace-O}
One has $\tr \Omega = x(2-x)S_{n-1}(z) -1.$
\end{lemma}

\begin{proof}
Since $\rho(w^{-1}) = \left[ \begin{array}{cc}
1+(2-s^{-2})u+u^2 & s-s^{-1}+su \\
(s^{-1}-s)u-su^2 & 1-s^2u \end{array} \right]$, by Lemma \ref{formulas} we have
$$\Delta = \left[ \begin{array}{cc}
P_{n-1}(z) - (1-s^2u) P_{n-2}(z) & (s-s^{-1}+su) P_{n-2}(z)\\
\big( (s^{-1}-s)u-su^2 \big) P_{n-2}(z) & P_{n-1}(z) - \big( 1+(2-s^{-2})u+u^2 \big) P_{n-2}(z) \end{array} \right].$$

By a direct computation we have $$\rho(a^{-1}(1-b)(1-a)) = \left[ \begin{array}{cc}
s + s^{-1} -2 + (s-1)u & s^{-1}-s^{-2}+u\\
su-s^2u & s+s^{-1}-2-su \end{array} \right].$$ Hence 
\begin{eqnarray*}
\tr \Omega &=& \tr \left( \rho \big( a^{-1}(1-b)(1-a)\big) \Delta \right)\\
&=& (2s + 2s^{-1}-4-u)P_{n-1}(z) + \big( 4-2s-2s^{-1}+(3-s^2-s^{-2})u+u^2 \big) P_{n-2}(z)\\
&=& (2s + 2s^{-1}-4-u) \big( P_{n-1}(z) - P_{n-2}(z) \big) + ((2-s^2-s^{-2})u+u^2)P_{n-2}(z)\\
&=& (2s + 2s^{-1}-4-u) S_{n-1}(z) + (z-2)P_{n-2}(z)\\
&=& (2s + 2s^{-1}-4-u) S_{n-1}(z) + S_{n-1}(z) - S_{n-2}(z) -1.
\end{eqnarray*}

Since $(s,u)$ satisfies $\phi_K(s,u)=0$, we have $S_n(z) = \left( u^2 -(u+1)(s^2+s^{-2}-3) \right)S_{n-1}(z)$. This implies that $S_{n-2}(z) = z S_n(z) - S_{n-1}(z) = (s^2+s^{-2}-1-u) S_{n-1}(z).$ Hence $$\tr \Omega = \big( 2s+2s^{-1}-s^2-s^{-2}-2 \big) S_{n-1}(z).$$ 
The lemma follows since $2s+2s^{-1}-s^2-s^{-2}-2 = x(2-x)$.
\end{proof}

We now complete the proof of Theorem \ref{main1}. Lemmas \ref{det-O} and \ref{trace-O} imply that $$\det \rho \left( \frac{\partial r}{\partial a} \right) = 1 + \det \Omega + \tr \Omega = (2-x)^2 \left( \frac{S_{n}(z) - S_{n-2}(z)-2}{z-2} \right) + x(2-x)S_{n-1}(z).$$
Since $\tau_{\rho}(E_K) = \det \rho \left( \frac{\partial r}{\partial a} \right) \big/ (2-x)$, we obtain the desired formula for $\tau_{\rho}(E_K)$.

\begin{remark}
In \cite{Mo}, Morifuji proved a similar formula for the twisted Alexander polynomial of twist knots for nonabelian representations.
\end{remark}

\subsection{Proof of Theorem \ref{main2}} Suppose $\rho: \pi_1(E_K) \to SL_2(\BC)$ is a nonabelian representation which extends to a representation $\rho: \pi_1(M) \to SL_2(\BC)$. Recall that $\lambda$ is the canonical longitude corresponding to the meridian $\mu=a$. If $\tr \rho(\lambda) \not= 2$, then by \cite{Ki2015} (see also \cite{Ki1994-fibered, Ki1994-8}) the Reidemeister torsion of $M$ is given by   
\begin{equation}
\label{Dehn}
\tau_{\rho}(M) = \frac{\tau_{\rho}(E_K)}{2-\tr \rho(\lambda)}.
\end{equation}

By Theorem \ref{main1} we have $\tau_{\rho}(E_K)=
( 2-x) \frac{S_n(z) - S_{n-2}(z) - 2}{z-2} + x S_{n-1}(z)$ if $x \not=2$. By Proposition \ref{longitude} we have $\tr\rho(\lambda) -2 = \frac{x^2}{u^{-2}(u+1)(x^2-4)-1}.$ Theorem \ref{main2} then follows from \eqref{Dehn}.

\end{document}